\pdfoutput=1
\documentclass[
  fontsize=10pt,
  paper=a4,
  german,USenglish
  ]{scrartcl}
\usepackage{babel}
\makeatletter
\defineshorthand{"-}{\penalty\@M-\futurelet\@let@token\goodhyphen@}
\newcommand*\goodhyphen@{%
  \ifmmode\errmessage{Special hyphenation in math mode}\fi
  \if\noexpand\@let@token\relax\else
  \@tempswafalse
  \catcode`\a=11
  \catcode`\*=12
  \ifcat\@let@token a\@tempswatrue\fi
  \ifcat\@let@token *\@tempswatrue\fi
  \if@tempswa
    {\setbox\z@\hbox{-\@let@token}\setbox\tw@\hbox{-\hskip\z@skip\@let@token}%
    \hskip\dimexpr\wd\z@-\wd\tw@\relax}%
  \fi
  \fi
}

\useshorthands{"}

\usepackage[utf8]{inputenc}
\usepackage{lmodern}
\normalfont
\usepackage[T1]{fontenc}
\areaset{141mm}{\dimexpr\topskip+49\baselineskip\relax}
\setkomafont{sectioning}{\bfseries}
\normallineskiplimit\normallineskip 
\usepackage{amsmath}
\usepackage{amsthm}
\usepackage{amsfonts}
\usepackage{amssymb}
\usepackage{ifpdf}
  \ifpdf
    \newcommand*\driver{}
  \else
    \newcommand*\driver{dvipdfmx}
  \fi
\usepackage[%
  pdftitle={Orientation reversal of manifolds},
  pdfauthor={Daniel Müllner},
  pdfsubject={},
  pdfdisplaydoctitle=true,
  pdfduplex=DuplexFlipLongEdge,
  colorlinks=false,
  pdfhighlight=/I,
  pdfborder={0 0 1},
  linkbordercolor={1 .8 .8},
  citebordercolor={.5 .9 .5},
  urlbordercolor={.5 .7 1},
  pdfpagemode=UseOutlines,
  bookmarksopen=true,
  bookmarksopenlevel=1,
  bookmarksdepth=2,
  breaklinks=true,
  unicode=true,
  \driver
]{hyperref}
\newcommand*\mynewtheorem[2]{%
  \newtheorem{#1}{#2}[section]%
  \expandafter\let\csname c@#1\endcsname=\c@thm
  \expandafter\def\csname#1autorefname\endcsname{#2}%
}
\renewcommand\thmhead[3]{%
  \thmname{#1}\thmnumber{\@ifnotempty{#1}{ }\@upn{#2}}%
  \thmnote{ {\the\thm@notefont#3}}}
\mynewtheorem{thm}{Theorem}
\mynewtheorem{lemma}{Lemma}
\mynewtheorem{example}{Example}
\mynewtheorem{cor}{Corollary}
\mynewtheorem{prop}{Proposition}
\theoremstyle{definition}
\newtheorem*{rem}{Remark}

\newenvironment{emptythm}{\list{}{\leftmargin2\parindent\rightmargin2\parindent}%
  \item\relax\itshape
}{\endlist}
\usepackage[lite]{amsrefs}
\renewcommand*\BibLabel{\Hy@raisedlink{\hyper@anchorstart{cite.\CurrentBib}\hyper@anchorend}[\thebib]}
\CheckCommand\PrintDOI[1]{DOI #1%
  \IfEmptyBibField{volume}{, (to appear in print)}{}}
\renewcommand*\PrintDOI[1]{\href{http://dx.doi.org/#1}{DOI #1}%
  \IfEmptyBibField{volume}{, (to appear in print)}{}}
\usepackage{booktabs}
\setbox0\hbox{$$}
\newlength\mathaxis
\usepackage[matrix,arrow,tips,curve,rotate,pdf]{xy}
  \SelectTips{cm}{10}
  \entrymodifiers={+!!<0pt,\mathaxis>}
\newcommand*\centercolon[1]{\vcenter{\mathsurround0pt\hbox{$#1:$}}}
\newcommand*\coloneq{\ensuremath{\mathrel{\mathpalette\centercolon=}}}

\newcommand*{\Cc}{\mathbb{C}}
\newcommand*{\Qq}{\mathbb{Q}}
\newcommand*{\Rr}{\mathbb{R}}
\newcommand*{\Zz}{\mathbb{Z}}

\newcommand*{\DeclareMathOrd}[2]{%
 \@ifdefinable{#1}{\DeclareRobustCommand{#1}{{\mathord{\operator@font#2}}}}}

 \newcommand*{\co}{\:{:}\;}
 \DeclareMathOperator\Ext{Ext}
 \DeclareMathOperator\Tor{Tor}
 \DeclareMathOperator\Hom{Hom}
 \DeclareMathOperator\Aut{Aut}
 \DeclareMathOperator\Out{Out}
 \DeclareMathOperator\SL{\mathit{SL}}
 \DeclareMathOperator\GL{\mathit{GL}}
 \DeclareMathOperator\SO{\mathit{SO}}
 \DeclareMathOperator\connsum{\mkern-1mu\#}

 \DeclareMathOrd\fr{fr}
 \DeclareMathOrd\id{id}
 \DeclareMathOrd\pt{pt}
 \DeclareMathOrd\cpt{cpt}
 \DeclareMathOrd\BO{{\mathit{BO}}}
 \DeclareMathOrd\EO{{\mathit{EO}}}
 \DeclareMathOrd\CP{\Cc P}
 \newcommand*\loc{_{\textup{(0)}}}
 \newcommand*\Nloc{_{N,\textup{(0)}}}
 \newcommand*{\ab}{^\mathit{ab}}
 \newenvironment{myalphenum}{%
 \begin{list}{(\alph{enumi})~\hfill}{\usecounter{enumi}%
 \labelwidth1.7em
 \labelsep\z@
 \leftmargin\dimexpr\labelwidth+\parindent\relax
 \rightmargin\parindent
 \topsep\medskipamount
 \itemsep\medskipamount
 \parsep\parskip
 \listparindent\parindent
 \partopsep\z@
 }}
 {\end{list}}
\addto\extrasUSenglish{%

}
\DeclareTextCommand{\textleq}{PU}{\9042\144}
\DeclareTextCommand{\textleq}{T1}{{\boldmath\ensuremath{\leq}}}
\DeclareTextCommand{\textgeq}{PU}{\9042\145}
\DeclareTextCommand{\textgeq}{T1}{{\boldmath\ensuremath{\geq}}}
\DeclareTextCommand{\textpm}{PU}{\9000\261}
\DeclareTextCommand{\textpm}{T1}{{\boldmath\ensuremath{\pm}}}
\DeclareTextCommand{\unicodespace}{PU}{ }
\DeclareTextCommand{\unicodespace}{T1}{\ensuremath{\mathrel{}{}}}
\newcommand*\Hopfian{Hopf\penalty\@M ian}
\makeatother

\begin{document}
\title{Orientation reversal of manifolds}
\author{\scshape Daniel Müllner}
\date{}
\maketitle
\begin{abstract}\noindent
We call a closed, connected, orientable manifold in one of the categories TOP, PL or DIFF \emph{chiral} if it does not admit an orientation-reversing automorphism and \emph{amphicheiral} otherwise. Moreover, we call a manifold \emph{strongly chiral} if it does not admit a self-map of degree $-1$. We prove that there are strongly chiral, smooth manifolds in every oriented bordism class in every dimension $\geq 3$. We also produce simply-connected, strongly chiral manifolds in every dimension $\geq\nolinebreak 7$. For every $k\geq 1$, we exhibit lens spaces with an orientation-reversing self-diffeomorphism of order $2^k$ but no self-map of degree $-1$ of smaller order.
\end{abstract}

\section{Introduction}
\let\savethethm\thethm
\renewcommand*\thethm{\Alph{thm}}
\let\savetheprop\theprop
\renewcommand*\theprop{\Alph{prop}}

In this paper, we study the question of whether a manifold admits an orientation-reversing self-map. Let $M$ be a closed, connected, orientable manifold in one of the categories TOP, PL or DIFF. We call $M$ \emph{chiral} if it does not admit an orientation-reversing automorphism in the respective category and \emph{amphicheiral} if it does. For the sake of clarity, we indicate the category by adverbs: e.\,g.\ a \emph{topologically chiral} manifold does not admit an orientation-reversing self-homeomorphism, whereas a \emph{smoothly amphicheiral} manifold is a differentiable manifold which admits an orientation-reversing self-diffeomorphism. We extend this by the notion of \emph{homotopical chirality/amphicheirality} when we consider homotopy self-equivalences. Chiral manifolds in the strongest sense do not admit self-maps of degree $-1$; we call them \emph{strongly chiral} and define \emph{weakly amphicheiral} as the opposite. The definition of chirality may be extended to non-closed manifolds and to oriented manifolds with several components, both of which are not in the focus of the present work.

Many familiar manifolds like spheres or orientable surfaces are smoothly amphicheiral: in these cases mirror-symmetric embeddings into $\Rr^n$ exist, and reflection at the “equatorial” hyperplane reverses the orientation. On the other hand, instances of strongly chiral manifolds have been known for many decades, e.\,g.\ the complex projective spaces $\CP^{2k}$ or some lens spaces in dimensions congruent 3 mod 4.

A basic question which was not fully answered for any type of chirality by the known examples is in which dimensions chiral manifolds exist. We solve this problem not only for all dimensions but also under the finer differentiation of oriented bordism. Trivially, a point is chiral, and all one- and two-dimensional manifolds are smoothly amphicheiral.

\begin{thm}\label{thm:A}
In every dimension $\geq 3$, every closed, smooth, oriented manifold is oriented bordant to a connected manifold of this type which is strongly chiral.
\end{thm}

In the proof, we first construct strongly chiral manifolds in every dimension $\geq 3$ and then modify these to obtain examples in all bordism classes. The manifolds which we construct in the first step are aspherical in odd dimensions and contain aspherical manifolds as factors in a cartesian product in many even dimensions. In these cases, it is the structure of the fundamental group and its endomorphisms which exclude self-maps of degree $-1$. We therefore ask for other obstructions and restrict the analysis to simply-connected manifolds.
\begin{thm}\label{thm:B}
In dimensions 3, 5 and 6, every simply-connected, closed, smooth manifold is smoothly amphicheiral. The analogous statements hold in the topological and the PL categories. A simply-connected, closed, topological 4-manifold is topologically amphicheiral if its signature is zero. If the signature is nonzero, the manifold is even strongly chiral.

In every dimension $\geq7$ there is a simply-connected, closed, smooth, strongly chiral manifold.
\end{thm}

In the amphicheiral case, it is also interesting to consider the minimal order of orientation-reversing maps. We give examples from the literature which admit an orientation-reversing diffeomorphism but none of finite order. We complement this with amphicheiral manifolds where the minimal order of an orientation-reversing map is finite but arbitrarily large:
\begin{thm}\label{thm:C}
For every positive integer $k$, there are infinitely many lens spaces which admit an orientation-reversing diffeomorphism of order $2^k$ but no self-map of degree $-1$ of smaller order.
\end{thm}

\subsection*{Outline}

In \autoref{sec:oldexamples}, we briefly survey known results and invariants which detect chirality. We also quote results which imply that homotopical and strong chirality coincide for all relevant manifolds in this paper. In \autoref{sec:newexamples}, we prove parts of \autoref{thm:A} and \autoref{thm:B} by constructing strongly chiral manifolds in all dimensions $\geq 3$ and simply-connected, strongly chiral manifolds in all dimensions $\geq 7$ except 9, 10, 13 and 17. These remaining four dimensions are dealt with in \autoref{sec:simplyconn}, which finishes the proof of \autoref{thm:B}. The proof of \autoref{thm:A} is completed in \autoref{sec:bordism} by extending the examples to all oriented bordism classes in dimensions $\geq 3$. Finally, \autoref{sec:minimalorder} is devoted to the proof of \autoref{thm:C}.

If not indicated otherwise, homology and cohomology are in this text always understood with integral coefficients.

\subsection*{Terminology}

Three distinct adjectives have been used in the literature to describe manifolds whose orientation can be reversed by a self-map: “symmetric” (by Rueff \cite{Rueff}*{page 162} and Kirby \cite{Kirby}*{Problem 1.23}), “amphicheiral” (by Siebenmann \cite{Siebenmann} and Saveliev \cite{Saveliev02} for 3-manifolds) and “reversible” (by Hirsch \cite{Hirsch}*{9.1.3, page 190} for surfaces). We chose the pair chiral/amphicheiral since it is justified by two parallels between knot theory and 3-manifold topology. The other two notions are dismissed since they either have now other meanings (asymmetric manifolds of Puppe \cite{Puppe}, symmetric spaces) or in the case of “reversibility” the existing concept in knot theory would suggest the wrong analogies.

First, if a 3-manifold admits a cyclic branched covering over an amphicheiral link it is amphicheiral. More precisely, we prove the following statement in \cite{Muellner}*{Proposition 19}.
\begin{prop}
Let $M$ be a triangulated 3-manifold and let $L$ be a piecewise linear link in $S^3$. Suppose that there is a PL-map $p\co M\to S^3$ which is a cyclic branched covering with branching set $L$. If the link $L$ is amphicheiral, then $M$ is PL-amphicheiral.
\end{prop}

The second parallel is the following: When a 3-manifold is formed by rational surgery on a link in $S^3$, the manifold with the opposite orientation is obtained by surgery on the mirror image of this link, with the negative surgery coefficients; see Saveliev \cite{Saveliev02}*{\textsection\,2.2, \textsection\,3.4}. Thus, surgery on an amphicheiral link which is appropriately labeled by rational surgery coefficients yields an amphicheiral 3-manifold. Since every 3-dimensional manifold has a unique smooth structure, it is not necessary to distinguish between topological, PL- and smooth amphicheirality here.

\subsection*{Acknowledgements}
This paper is a condensed version of a large part of my doctoral thesis written under the supervision of Prof.\ Matthias Kreck. I would like to thank him for his constant support during my doctoral studies. I also want to thank Prof.\ Shmuel Weinberger for several very useful comments and suggestions.

\let\thethm\savethethm
\let\theprop\savetheprop

\section{Known examples and obstructions}\label{sec:oldexamples}

In this section, we briefly survey known results and invariants which detect chirality. The examples show that not all kinds of chirality coincide.

\begin{itemize}
\item
A manifold with nonzero \emph{signature} does not admit a self-map of negative degree, in particular it is strongly chiral. Indeed, the \textit{intersection form} of a closed, oriented, $4k$-dimensional manifold is isomorphic to its negative if and only if its signature is zero; see Milnor and Husemoller \cite{MH}*{Theorem 5.3}. In contrast, the antisymmetric intersection form of a $(4k+2)$-dimensional manifold is always isomorphic to its negative \cite{MH}*{Corollary 3.5}.

\item
The \textit{linking form} of a closed, oriented, $(2k-1)$-dimensional manifold gives obstructions to amphicheirality if $k$ is even. For example, a $(4n-1)$-dimensional manifold with $H_{2n-1}(M)\cong\Zz/t$ is strongly chiral if $-1$ is not a quadratic residue modulo $t$. If $k$ is odd, the linking form is always isomorphic to its negative; see Wall \cite{Wall62}*{Lemma 4(ii)}.

\item
Explicit \emph{classification results} up to \textit{oriented} maps (diffeomorphisms, homotopy equivalences etc.)\ give information about chirality. For example, lens spaces have been classified up to orientation-preserving homotopy equivalence and homeomorphism/PL-equivalence/diffeomorphism; see Milnor \cite{Milnor66}*{\textsection\,12} and Lück \cite{Lueck}*{\textsection\,2.4}. There are lens spaces like $L_5(1,1)$ which are topologically chiral but homotopically amphicheiral.

\item
A manifold with nonzero \textit{Pontrjagin numbers} is topologically chiral by work of Novikov \cite{Novikov}.

\item
A closed, smooth, oriented manifold whose Pontrjagin numbers are zero is oriented bordant to a manifold with an orientation-reversing smooth involution (see Kawakubo \cite{Kawakubo} and Rosenzweig \cite{Rosenzweig}*{page 5, lines 13--20}). This smoothly amphicheiral representative can be chosen to be connected.

\item
A closed, smooth, oriented manifold is oriented bordant to its negative if and only if all its Pontrjagin numbers vanish; see Wall \cite{Wall60}.

\item
Many \textit{homotopy spheres} are examples of topologically amphicheiral but smoothly chiral manifolds; see Kervaire and Milnor \cite{KM}. Since the inverse of an element in the group $\theta_n$ of exotic $n$-spheres is given by the manifold with the opposite orientation, and for example we have $\theta_7\cong \Zz/28$, there are 13 pairs of smoothly chiral, homotopy 7-spheres. Obviously, these manifolds are topologically amphicheiral. The standard sphere and the exotic sphere of order 2 in $\theta_7$ are smoothly amphicheiral.

\item
In low-dimensional topology, more specialized invariants are available. The Casson invariant is a $\Zz$-valued homeomorphism invariant for oriented integral homology 3-spheres, which reverses its sign with the orientation; see Saveliev \cite{Saveliev99}*{especially Chapter 12}. For example, the Poincaré homology sphere has Casson invariant $-1$ \cite{Saveliev99}*{\textsection\,17.5} and is therefore topologically chiral.

A concept which produces chiral 3-manifolds in abundance, is homology bordism \cite{Saveliev99}*{\textsection\,11.4}. Since the homology bordism group $\Theta^3_\Zz$ of oriented integral homology 3-spheres contains a free abelian group of infinite rank, this yields a countable infinite number of homology 3-spheres, all of which are topologically chiral.

\item
The specialized techniques in 4-dimensional topology can be used to show smooth chirality. Based on Donaldson invariants, Kotschick exhibits a smoothly chiral, simply-connected 4-manifold (more precisely a minimal compact complex surface of general type) with signature 0 \cite{Kotschick92}*{Theorem 3.7, Remark 3.9}, in contrast to the topological amphicheirality of such manifolds by Freedman's results \cite{FQ}.
\end{itemize}

A closed, connected, orientable manifold $M$ is called \emph{\Hopfian} if every map $f\co M\to M$ of degree one is a homotopy equivalence. Since a self-map $f$ is a homotopy equivalence if and only if $f^2$ is, also a self-map of degree $-1$ of a \Hopfian{} manifold is a homotopy equivalence. Every manifold with finite fundamental group is \Hopfian. This has an elementary proof using Whitehead's theorem \cite{Whitehead1949}*{Theorem 3} and the Umkehr homomorphism in homology. Also note that a map of degree $\pm1$ between closed, connected, orientable manifolds induces a surjection on the fundamental groups \cite{Hausmann}. Orientable hyperbolic manifolds are also \Hopfian{} according to Sela \cite{Sela}.

In summary, the notions of weak amphicheirality and homotopical amphicheirality coincide for \Hopfian{} manifolds. Likewise, homotopical chirality implies strong chirality for \Hopfian{} manifolds. We will use this several times, as all relevant manifolds in this paper are \Hopfian.

\section{Strongly chiral manifolds in every dimension \textgeq\unicodespace3}\label{sec:newexamples}

In this section, we prove the parts of \autoref{thm:A} and \autoref{thm:B} which make sense to be dealt with together. In \autoref{oddexsec}, a series of strongly chiral manifolds in every odd dimension $\geq3$ is constructed. In \autoref{prodchiral}, we use cartesian products of chiral manifolds to produce even-dimensional chiral manifolds in all dimensions $n\equiv 2$ mod 4, $n\geq 6$. In dimensions congruent 0 modulo 4, many examples are known, like complex the complex projective spaces $\CP^{2k}$, or any other manifold with nonzero signature. We also obtain simply-connected, strongly chiral manifolds in all dimensions $\geq 7$ except 9, 10, 13 and 17.

Shmuel Weinberger pointed out a different way of proving \autoref{thm:A} to the author, which yields strongly chiral manifolds in all possible dimensions as a quick corollary, although not in a constructive and rather elementary manner: Belolipetsky and Lubotzky proved in \cite{BL} that every finite group can be realized as the isometry group of a compact, hyperbolic manifold of arbitrary dimension $\geq2$. The manifolds of Belolipetsky and Lubotzky are in fact orientable, and they are \Hopfian{} according to Sela \cite{Sela}. Thus by Mostow rigidity, every self-map of degree $\pm1$ is homotopic to an isometry if the dimension is at least 3. Now choose any finite group of odd order as the isometry group. This way, one obtains many strongly chiral, hyperbolic manifolds in every dimension $\geq 3$. Since the manifolds are aspherical, \autoref{oddbordant} can be used to cover all bordism classes in dimensions $\geq 3$.

\subsection{Strongly chiral manifolds in every odd dimension \textgeq\unicodespace3}\label{oddexsec}

Examples of strongly chiral manifolds in odd dimensions will be provided by mapping tori of $n$-dimensional tori $T^n$. We can exclude orientation-reversing maps by studying the endomorphisms of the fundamental group.

Let $f\co T^n\to T^n$ be an orientation-preserving diffeomorphism. The mapping torus of $f$,
\[
 M_f\coloneq T^n\times [0,1] \ \bigm/\  (x,0)\sim (f(x),1).
\]
is a fibre bundle over $S^1$ with fibre $T^n$. According to the long exact sequence of homotopy groups, $M_f$ is aspherical, and its fundamental group is a semidirect product $\pi_1(T^n)\rtimes\pi_1(S^1)\cong \Zz^n\rtimes\Zz$. Conjugation by lifts of elements in $\pi_1(S^1)$ determines a homomorphism $\psi\co\pi_1(S^1)\to \Out(\pi_1(T^n))$, or isomorphically $\Zz\to\GL_n(\Zz)$. In the present case, $\psi$ is given by $\psi(\id_{S^1})=f_*\co \pi_1(T^n)\to \pi_1(T^n)$, which can be seen from the attaching maps of the 2-cells in an appropriate CW-decomposition of $M_f$.

In the following, we require that $f_*-\id$ is an isomorphism on $\pi_1(T^n)$. Under these circumstances, we show in \cite{Muellner}*{Corollary 26} that $\pi_1(S^1)$ is the abelianization of $\pi_1(M_f,*)$ and $\pi_1(T^n)$ its commutator subgroup. Therefore, every endomorphism $J$ of $\pi_1(M_f)$ induces endomorphisms $J\ab$ and $J^{(1)}$ on $\pi_1(S^1)$ and $\pi_1(T^n)$ respectively. These two induced endomorphisms must be compatible with the action $\psi$:
\begin{equation}\label{condition}
 \psi(J\ab(h))(J^{(1)}(n)) = J^{(1)}(\psi(h)(n))
 \quad\text{for all $n\in\pi_1(M_f)$, $h\in\pi_1(S^1)$}
\end{equation}

Consider a self-map $K\co M_f\to M_f$. Since $[M_f,S^1]\cong H^1(M_f)\cong\Zz$, we can always complete the following diagram so that it commutes up to homotopy:
\[
  \xymatrix{T^n \ar[r]\ar@{.>}[d]^{K''} & M_f\ar[r]\ar[d]^K & S^1\ar@{.>}[d]^{K'} \\
  T^n\ar[r] & M_f\ar[r] & S^1 }
\]
This means that $K$ is homotopic to a fibre-preserving map. We can thus employ the naturality of the Serre spectral sequence, which has the highest $E_2$-term $E_2^{1,n}=H^1(S^1;H^n(T^n))\cong H^1(S^1)\otimes H^n(T^n)$ (constant coefficients since $f$ acts trivially on $H^n(T^n)$). We deduce from the spectral sequence in \cite{Muellner}*{Lemma 28} that the degree of $K$ is given by the product $\deg K'\cdot \deg K''=\deg K'\cdot \det(K''_*\co H_1(T^n)\to H_1(T^n))$.

Having chosen a basis for $H_1(T^n)=\pi_1(T^n)\cong \Zz^n$, every matrix $F\in\SL(n,\Zz)$ can be realized as the map on $\pi_1(T^n)$ which is induced by an orientation-preserving diffeomorphism $f\co T^n\to T^n$. Hence, we can construct a chiral $(n+1)$-manifold under the following circumstances:

\begin{lemma}\label{matrixeq}
Suppose there is a matrix $F\in\SL(n,\Zz)$ such that
\begin{myalphenum}
\item
$\det(F-\id)=\pm1$,

\item
the equation $FG=GF$ has no solution $G\in\GL(n,\Zz)$, $\det G=-1$,

\item
the equation $F^{-1}G=GF$ has no solution $G\in\SL(n,\Zz)$.
\end{myalphenum}
Then a mapping torus $M_f$ with $f\co T^n\to T^n$ realizing $F$ on $\pi_1(T^n)\cong\Zz^n$ is strongly chiral.
\end{lemma}

\begin{proof}
This lemma is a reformulation of the previous considerations, in particular the condition given by equation \eqref{condition}. The correspondence between the notation here and in the previous paragraphs is
\[\tabskip0pt plus 1fill
  \halign to \displaywidth{\tabskip0pt\hfil$\displaystyle#$\hfil\hskip3em&\hfil$\displaystyle#$\hfil\tabskip0pt plus1fill&#\tabskip0pt\cr
  F\leftrightarrow f_*=\psi(\id)\co\pi_1(T^n)\to\pi_1(T^n), &  F^{-1}\leftrightarrow \psi(-\id),\cr
  \noalign{\penalty10000\vskip2\jot}
  J\ab(\id)=K'_*(\id)=\pm\id\in\pi_1(S^1),&  G\leftrightarrow J^{(1)}=K''_*\co \pi_1(T^n)\to \pi_1(T^n).&\qedhere\cr
 }
\]
\vskip-\prevdepth
\vskip-\belowdisplayskip
\hrule height0pt
\end{proof}

We show that every matrix $F\in M(n\times n;\Zz)$, for even $n$, with characteristic polynomial
\[
 \chi_F(X)=X^{n}-X+1
\]
fulfills the lemma. A matrix with the required characteristic polynomial is given, e.\,g., by the following scheme:
\[
 \vbox{\hbox{$\displaystyle
 \xy
 <.72pc,0pc>:
 0, *h!!<0pt,\mathaxis>!R{F\coloneq{}\,},
 (0,-5);(1,-5):
 (1,1) *!!<0pt,\mathaxis>{-1},
 (3,1) *!!<0pt,\mathaxis>{1},
 (7,1) *!!<0pt,\mathaxis>{0},
 (1,6) *!!<0pt,\mathaxis>{0},
 (6,6) *!!<0pt,\mathaxis>{I_{n-1}},
 (0,0), {\ar@{-} '(10,0) '(10,10) '(0,10) (0,0)},
 (0,2), {\ar@{-} (10,2)},
 (2,0), {\ar@{-} (2,10)},
 (4,0), {\ar@{-} (4,2)},
 \endxy$}
 \hrule height0pt}
\]

The value $\chi_F(0)=1$ guarantees $F\in\SL(n,\Zz)$, while $\chi_F(1)=1$ ensures condition (a). Next we show that there is no solution to equation (b). By basic linear algebra, the eigenvalues of $F$ are distinct pairs of conjugate complex numbers $(\lambda_i,\bar\lambda_i)$, $\lambda_i\in\Cc\setminus\Rr$. Thus, $F$ is diagonalizable over $\Cc$ and its determinant is given by
\[
 \det F=\prod_{i=1}^{n/2}\lambda_i\bar\lambda_i = \prod_{i=1}^{n/2}|\lambda_i|^2.
\]

Suppose now that $FG=GF$, i.\,e.\ $G$ lies in the centralizer of $F$. Since $F$ has distinct eigenvalues, linear algebra again tells us that $G$ is a polynomial expression $p(F)$, possibly with rational coefficients. Hence, $G$ is also diagonalizable, and its determinant is given by the products of polynomials
\[
 \det G=\prod_{i=1}^{n/2}p(\lambda_i)\overline{p(\lambda_i)}
 =\prod_{i=1}^{n/2}|p(\lambda_i)|^2
 \geq 0.
\]

This contradicts $\det G=-1$, hence equation (b) has no solution.

Now we show that there is no solution to equation (c). Suppose $GF=F^{-1}G$ for some $G\in\SL(n,\Zz)$. Then $F$ and $F^{-1}$ have the same eigenvalues. However, this possibility can be easily excluded from the characteristic polynomial for $n>2$; see \cite{Muellner}*{page 36}. In the case $n=2$, the two eigenvalues of $F$ are in fact inverses of each other. However, in this case $G$ is necessarily of the form
\[
 G=\begin{pmatrix}
 a &b\\a+b&-a
 \end{pmatrix}
 \quad\text{with $a,b\in\Zz$,}
\]
and its determinant $\det G=-(a^2+ab+b^2)=-\frac12(a^2+(a+b)^2+b^2)$ is always negative. Thus, there is no solution to (c) in any case.\qed

\subsection{Products of strongly chiral manifolds}\label{prodchiral}

The chiral manifolds that we constructed in odd dimensions can also be used to obtain examples in even dimensions. We show that under certain conditions, products of strongly chiral manifolds are again strongly chiral. Moreover, we use this method to prove the existence of simply-connected, strongly chiral manifolds in sufficiently high dimensions.

\begin{prop}\label{products}
Let $\Sigma$ be a rational homology sphere and $M$ a closed, connected, orientable manifold of the same dimension which is not a rational homology sphere. The product $\Sigma\times M$ is strongly chiral if and only if both factors are.
\end{prop}

\begin{example}\label{6mod8}
Let $\Sigma$ be a lens space of dimension $n\equiv3$ \textup{mod 4} with fundamental group of order $t\equiv3$ \textup{mod 4}, and let $M$ be a strongly chiral mapping torus of the same dimension, as constructed in \autoref{oddexsec}. This yields examples of strongly chiral manifolds in each dimension congruent 6 modulo 8.
\end{example}

\begin{proof}
The lens space $\Sigma$ is strongly chiral according to the second item in the list in \autoref{sec:oldexamples} since t contains a prime factor congruent 3 mod 4. Furthermore, we have $H_1(M)=\pi_1(M)\ab\cong\Zz$, so $M$ is not a rational homology sphere.
\end{proof}

\begin{proof}[Proof of \autoref{products}]
The “only if” part is obvious, hence we deal with the necessity of the condition. Let $n$ be the dimension of $\Sigma$ and $M$. In the following, cohomology is understood with rational coefficients. By the Künneth theorem, we have
\[
 H^n(\Sigma\times M)\cong H^n(\Sigma)\oplus H^n(M)\cong\Qq^2.
\]

Consider the cohomology classes $[\Sigma]^*\in H^n(\Sigma)$ and $[M]^*\in H^n(M)$ which are Kronecker dual to the fundamental classes $[\Sigma]$, $[M]$. Let $T\co\Sigma\times M\to\Sigma\times M$ be a continuous map. The effect on $H^n(\Sigma\times M)$ is given (with respect to the basis $[\Sigma]^*$, $[M]^*$) by an integral matrix
\[
 \begin{pmatrix}
   a & b \\ c & d
 \end{pmatrix}.
\]

Since $[\Sigma]^*\cup[M]^*$ is a generator of $H^{2n}(\Sigma\times M)\cong\Qq$, the mapping degree of $T$ is given by $ad+(-1)^nbc$. The coefficient $b$ can be recovered as the mapping degree of $p_M\circ T\circ i_\Sigma\co\Sigma\to M$, where the maps $i_\Sigma$ and $p_M$ are the usual inclusion and projection.

By a theorem of Hopf following from the Umkehr homomorphism \cite{Hopf}*{Satz IIIa}, every map $\Sigma\to M$ must have degree zero since at least one Betti number of $\Sigma$ is smaller than that of $M$. Thus, the degree of $T$ is equal to the product $ad$. Since neither of the factors can be $-1$ by assumption, $T$ cannot reverse the orientation.
\end{proof}

If $\Sigma$ and $M$ have different dimensions, chirality of the product is even easier to prove since only the Künneth theorem is needed, not the argument with the Umkehr map.

\begin{prop}[\cite{Muellner}*{Theorem 34}]\label{products2}
Let $\Sigma$ be a rational homology sphere of dimension $s$ and $M$ a closed, connected, orientable manifold of different dimension $m\neq s$. Also require that $H^s(M;\Qq)=0$. The product $\Sigma\times M$ is strongly chiral if and only if both factors are.
\end{prop}

\begin{example}\label{2mod4}
Let $L=L_1\times\ldots\times L_k$ be a product of lens spaces of pairwise different dimensions. Then $L$ is strongly chiral if and only if this holds for each single factor.
\end{example}

This follows from \autoref{products2} by induction. Note that the condition in the example can be easily tested by the oriented homotopy classification of lens spaces \cite{Milnor66}*{\textsection\,12.1}. Since strongly chiral lens spaces occur in all dimensions $n\equiv3$ mod 4, this yields in particular strongly chiral manifolds in each dimension congruent 2 mod 4 starting from $10=7+3$. This finishes the construction of strongly chiral manifolds in all dimensions $\geq 3$.

We now want to apply this approach to simply-connected, strongly chiral manifolds and cover as many dimensions as possible. As “starting dimension” we cannot use 3 as in the examples with nontrivial fundamental group. Instead, we use highly connected, strongly chiral manifolds in dimensions congruent 3 mod 4 from 7 on as the “building blocks”. Suitable manifolds are e.\,g.\ linear $S^{2k-1}$ bundles over $S^{2k}$ with Euler class $6[S^{2k}]^*$. These are $(2k-2)$-connected closed manifolds with first nontrivial homology group $H_{2k-1}\cong\Zz/6$. Since $-1$ is not a quadratic residue modulo 6, the linking form forbids orientation reversal. (Note that an odd Euler class can only be achieved in dimensions 7 and 15, due to the Hopf invariant one problem. The Euler class 6 is always possible: pull back the sphere bundle of the tangent bundle $TS^{2k}$ by a smooth map of degree 3.)

This provides us with simply-connected, strongly chiral manifolds in every dimension $n\equiv3$ mod 4 starting from $n=7$. For completing the examples in higher dimensions, we need simply-connected, strongly chiral 7-manifolds that are \emph{not} rational homology spheres. Explicitly, we define $N_1$ to be the connected sum of $S^3\times S^4$ with a 7-manifold $M^7$ as constructed above. Similarly, $N_2\coloneq(S^2\times S^5)\connsum M^7$. They are strongly chiral because of the linking form.

\begin{cor}
In every dimension $n\equiv 2$ \textup{mod 4} starting from 14, there is a simply-connected, strongly chiral manifold.
\end{cor}

\begin{proof}
In dimension 14, take the product of a simply-connected, 7-dimensional rational homology sphere $M^7$ from the previous paragraphs with either $N_1$ or $N_2$ and apply \autoref{products}. In higher dimensions, use products of two rational homology spheres of different dimensions congruent 3 mod 4 and apply \autoref{products2}.
\end{proof}

\begin{cor}\label{simpconn1mod4}
In every dimension $n\equiv 1$ \textup{mod 4} starting from 21, there is a simply-connected, strongly chiral manifold.
\end{cor}

\begin{proof}
From dimension 25 on, we can take the product of the 14-dimensional manifold $M^7\times N_2$ from the previous corollary with another strongly chiral rational homotopy sphere. Note that, according to the rational Künneth theorem, $M^7\times N_2$ has no rational homology in degree 11 (and of course not in higher degrees congruent 3 mod 4). Thus, \autoref{products2} applies.

For dimension 21, consider $M^{21}\coloneq M^7\times N_1\times N_2$. Here, we can argue in a similar way as in the proof of \autoref{products}; we only have to consider products of three manifolds instead of two (see \cite{Muellner}*{Corollary 40}).
\end{proof}

\section{Amphicheirality and chirality of simply-connected manifolds}\label{sec:simplyconn}

In this section, we prove \autoref{thm:B}. The statements in dimensions 3 to 5 follow from the classification results for simply-connected manifolds in these dimensions more or less immediately. A closed, simply-connected 3-manifold is diffeomorphic to the 3-sphere by the Poincaré conjecture, as proved by Perelman (see Morgan and Tian \cite{MT}), and is hence smoothly amphicheiral. In dimensions 4 to 6, the classification results were obtained respectively by Freedman \cite{FQ}, Barden \cite{Barden} and Zhubr \cite{Zhubr}. In \cite{Muellner}*{\textsection\,4.1}, we provide the details to conclude amphicheirality from the classifications. The classification statement in dimension 6 is considerably more complicated than in the other dimensions, and we must add some smaller observations to prove amphicheirality. Denote the second homology group of the 6-manifold $M$ in question by $G$. All the invariants of Zhubr can be expressed as numbers or in terms of $G$. One invariant is for example the image of the fundamental class $\mu=f_*[M]\in H_6(K(G,2))$ under the first nontrivial Postnikov map $f\co M\to K(G,2)$. A necessary condition which has to be verified is that $\mu$ can be sent to $-\mu$ by an automorphism of $G$. Indeed, we show that for a finitely generated abelian group $G$, every element of $H_6(K(G,2))$ is sent to its inverse by the automorphism $-\id$ on $G$. We also show that all other invariants are sent to the values for $M$ with the opposite orientation under the maps which are induced by $-\id$ on $G$.

The main result of this section is the second part of \autoref{thm:B}, where we claim the existence of simply-connected, strongly chiral manifolds in every dimension $\geq 7$. We constructed such manifolds in all dimensions congruent 1, 2 and 3 mod 4, except 9, 10, 13 and 17, in \autoref{prodchiral}. The complex projective spaces $\CP^{2k}$ clearly provide simply-connected examples in dimensions congruent 0 mod 4. We prove that strongly chiral, simply-connected manifolds exist in the remaining dimensions in \autoref{1017} (dimensions 10 and 17) and \autoref{913} (dimensions 9 and 13).

These proofs are less constructive since the problem is more complicated than for an arbitrary fundamental group. In fact, we divide the problem into two manageable parts: (1) identify an obstruction to amphicheirality in the partial homotopy type and (2) realize the obstruction by a simply-connected manifold.

The first step is done with the help of the Postnikov tower: By its functoriality up to homotopy (see Kahn \cite{Kahn}), every self map of a space $M\to M$ induces a map $P^k\to P^k$ on its Postnikov approximations $p\co M\to P^k$ ($k=1,2,\ldots$) such that the diagrams
\[
 \xymatrix{M \ar[r]^p\ar[d] & P^k \ar[d] \\ M \ar[r]^p & P^k}
\]
commute up to homotopy. The same holds for rational Postnikov approximations. Whenever a homology class $m\in H_n(P^k)$ cannot be sent to its negative by any self-homotopy equivalence of $P^k$, also the preimages $(p_*)^{-1}(m)$ in the homology of $M$ cannot be reversed.

To exploit this property, we construct an appropriate finite tower of principal $K(\pi, n)$-fibrations (or simply a single stage) as a candidate for the Postnikov tower and fix an element in the integral homology of one of the stages that is to be the image of the fundamental class of the manifold. Then we prove that (by the mechanism that lies in the particular construction) this homology class can never be mapped to its negative under any self-map of a single Postnikov stage or of the partial Postnikov tower.

In the arguments given below, we do not relate the induced maps between successive stages in the tower of fibrations by functoriality. (For $P^3\to P^2$, it is done, though, explicitly and elementary in the proof of \autoref{diaglemma}, and for $P^4\loc\to P^3\loc\to P^2\loc$ we will refer to the properties of minimal models instead of spaces.) The full naturality statement involving all Postnikov stages at once is Kahn \cite{Kahn}*{Theorem 2.2}.

In the second step, the obstruction is realized by proving that there is indeed an $n$-dimensional manifold $M$ with the correct partial homotopy type (i.\,e.\ a map $p\co M\to P^k$ which is a $(k+1)$-equivalence) and the correct image $p_*[M]=m$ of the fundamental class in the Postnikov approximation. This step involves bordism computations and surgery techniques.

\begin{rem}
One can also define chirality in the algebraic setting of Poincaré duality algebras (see Papadima \cite{Papadima} or Meyer and Smith \cite{MeyerSmith}), possibly with some extra structure. The manifolds we construct in the following sections indicate interesting types of Poincaré duality algebras to consider: e.\,g.\ Poincaré duality algebras over $\Zz/p$ ($p>2$) with an action of the mod-$p$ Steenrod algebra, rational differential algebras and their minimal models and also integral lattices in rational Poincaré duality algebras. A detailed treatment of such algebraic generalizations is outside the scope of this paper, particularly since realizing the algebraic obstructions by manifolds is an essential part of this work.
\end{rem}

\subsection{Dimensions 10 and 17}\label{1017}

\begin{thm}\label{10dim1conn}
There exists a simply-connected, closed, smooth, strongly chiral, 10"-dimensional manifold.
\end{thm}

\begin{proof}
We follow the outline described above and choose the Eilenberg-MacLane space $K(\Zz/3,3)$ as a candidate for the first nontrivial, single Postnikov stage. From the structure of the subalgebra $H^*(K(\Zz/3,3);\Zz/3)$ of the Steenrod algebra $\mathcal A_3$ (see Hatcher \cite{HatcherAT}*{\textsection\,4L}) and the universal coefficient theorems, we can deduce \autoref{cohop} below. As usual, $P^1\co H^3(\mathord{\,-\,};\Zz/3)\to H^7(\mathord{\,-\,};\Zz/3)$ denotes the first Steenrod power operation, and $\rho_3$ is the reduction of integral coefficients modulo 3.

\begin{lemma}[\cite{Muellner}*{Lemma 48}]\label{cohop}
Let $\iota\in H^3(K(\Zz/3,3);\Zz/3)$ denote the canonical generator. There is a homology class $m\in H_{10}(K(\Zz/3,3))$ such that $\langle\iota\cup P^1\iota,\rho_3m\rangle$ is nonzero in $\Zz/3$.
\end{lemma}

We claim that there is a 10-dimensional, 2-connected manifold $M$ with first Postnikov approximation $f\co M\to K(\Zz/3,3)$ such that the image $f_*[M]$ of the fundamental class is $m$. Let $i\coloneq f^*\iota\in H^3(M;\Zz/3)$. If $T\co M\to M$ is a homotopy equivalence, $i$ is multiplied by some factor $k\in\Zz/3$: $T^*i=k\cdot i$. Then we have
\[
 \vbox{\openup\jot\halign{\hfill$\displaystyle#{}$&$\displaystyle{}#$\hfill\cr
 \deg T\cdot \langle i\cup P^1i,\rho_3[M]\rangle&=\langle i\cup P^1i,T_*\rho_3[M]\rangle=\langle (T^*i)\cup P^1(T^*i),\rho_3[M]\rangle\cr
 &=k^2\langle i\cup P^1i,\rho_3[M]\rangle.\crcr
 }}
\]
By \autoref{cohop}, the Kronecker product $\langle i\cup P^1i,\rho_3[M]\rangle$ is nonzero. Thus, $\deg T\equiv k^2$ mod 3, and this is never congruent $-1$, so $T$ cannot reverse the orientation.

Now we prove the existence of a 10-manifold with first Postnikov approximation $K(\Zz/3,3)$ and the correct image of the fundamental class. For the sake of simplicity, we look for a \emph{framed} manifold $M$. This task can be formulated as a bordism problem: Show that there is an element $(M,f)$ in the 10-dimensional singular framed bordism group $\Omega^{\mathrm{fr}}_{10}(K(\Zz/3,3))$ that maps to $m$ under the Thom homomorphism
\[
 \vbox{\openup\jot\halign{\hfill$\displaystyle#{}$&$\displaystyle{}#$\hfill\cr
  \Omega^{\mathrm{fr}}_{10}(K(\Zz/3,3)) &\to  H_{10}(K(\Zz/3,3)) \cr
  (M,f) & \mapsto f_*[M].\crcr
 }}
\]

The Thom homomorphism factors through the edge homomorphism $E^\infty_{10,0} \hookrightarrow E^2_{10,0}$ in the Atiyah-Hirzebruch spectral sequence for the homology theory $\Omega^{\mathrm{fr}}_*$; see Conner \cite{Conner}*{\textsection\,1.7} and Kochman \cite{Kochman}*{\textsection\,4.2}:
\[
  \Omega^{\mathrm{fr}}_{10}(K(\Zz/3,3)) \twoheadrightarrow  E^\infty_{10,0} \hookrightarrow E^2_{10,0}\cong H_{10}(K(\Zz/3,3))
\]
Thus, it is sufficient to prove surjectivity of the edge homomorphism.

In the following, we restrict ourselves to the reduced bordism group $\widetilde\Omega^{\mathrm{fr}}_{10}(K(\Zz/3,3))\subset \Omega^{\mathrm{fr}}_{10}(K(\Zz/3,3))$. This is no limitation: we actually produce manifolds which are nullbordant in $\Omega^{\mathrm{fr}}_{10}$.

The $E^2$-terms in the reduced Atiyah-Hirzebruch spectral sequence are given by $E^2_{r,s}\cong\widetilde H_r(K(\Zz/3,3);\Omega^{\mathrm{fr}}_s)$. Since there are no bordism groups of negative degree, the Atiyah-Hirzebruch spectral sequence is located in the first quadrant, and we have $E^\infty_{r,0}=E^{r+1}_{r,0}$. It is sufficient to show that all the intermediate inclusions $E^{i+1}_{10,0}=\ker d_i\subseteq E^i_{10,0}$ are in fact bijections, i.\,e.\ we want to show that all differentials starting from $E^2_{10,0}$ are zero. \hyperref[fig:AHSS:top]{\autoref*{fig:AHSS}} shows the relevant part of the Atiyah-Hirzebruch spectral sequence for $\widetilde\Omega^{\mathrm{fr}}_{10}(K(\Zz/3,3))$. It reveals that all terms $E^2_{r,9-r}$ on the 9-line are zero. Thus, the Thom homomorphism is surjective.
\begin{figure}\phantomsection\label{fig:AHSS:top}
\centerline{$
 \xy
 <2.3pc,0pt>:<0pt,1.75pc>::
 (-.5,-.5);p \ar (10.5,-.5)="id"
   \POS "id"+/r2\jot/ *!!<0pt,\mathaxis>!L{r},
 p \ar (-.5,10.5)="id"
   \POS "id"+/r2\jot/ *!!<0pt,\mathaxis>!L{s},
 0 *!!<0pt,\mathaxis>{0},
 +(1,0) *!!<0pt,\mathaxis>{0},
 +(1,0) *!!<0pt,\mathaxis>{0},
 +(1,0) *!!<0pt,\mathaxis>{\Zz/3},
 +(1,0) *!!<0pt,\mathaxis>{0},
 +(1,0) *!!<0pt,\mathaxis>{0},
 +(1,0) *!!<0pt,\mathaxis>{0},
 +(1,0) *!!<0pt,\mathaxis>{(\Zz/3)^2},
 +(1,0) *!!<0pt,\mathaxis>{\Zz/3},
 +(1,0) *!!<0pt,\mathaxis>+{0}="90",
 +(1,0) *!!<0pt,\mathaxis>+{\Zz/3}="100",
 "90";
 (8,1) *!!<0pt,\mathaxis>+{0} **\dir{-};
 (7,2) *!!<0pt,\mathaxis>+{0} **\dir{-};
 (6,3) *!!<0pt,\mathaxis>+{0} **\dir{}="63" ?<;?>>> **\dir{-},
 "63";
 (5,4) *!!<0pt,\mathaxis>+{0} **\dir{-};
 (4,5) *!!<0pt,\mathaxis>+{0} **\dir{-};
 (3,6) *!!<0pt,\mathaxis>+{0} **\dir{-};
 (2,7) *!!<0pt,\mathaxis>+{0} **\dir{-};
 (1,8) *!!<0pt,\mathaxis>+{0} **\dir{-};
 (0,9) *!!<0pt,\mathaxis>+{0} **\dir{-};
 (10,10) *!UR{E^2_{r,s}\cong\widetilde{H}_r(K(\Zz/3,3);\Omega^{\mathrm{fr}}_s)},
 "100"!C!R(.5) \ar_-{d_i} "63"!C
  \POS
  (-1.5,0) *!!<0pt,\mathaxis>{\Zz},
 +(0,1) *!!<0pt,\mathaxis>{\Zz/2},
 +(0,1) *!!<0pt,\mathaxis>{\Zz/2},
 +(0,1) *!!<0pt,\mathaxis>{\Zz/24},
 +(0,1) *!!<0pt,\mathaxis>{0},
 +(0,1) *!!<0pt,\mathaxis>{0},
 +(0,1) *!!<0pt,\mathaxis>{\Zz/2},
 +(0,1) *!!<0pt,\mathaxis>{*},
 +(0,1) *!!<0pt,\mathaxis>{*},
 +(0,1) *!!<0pt,\mathaxis>{*},
 +(0,1) *!!<0pt,\mathaxis>!C{\Omega^{\mathrm{fr}}_*},
 (-2,9.5);p+(1,0) **\dir{-},
 (-2,-.5) \ar@{-} '(-2,10.5) '(-1,10.5) '(-1,-.5) (-2,-.5)
 \endxy
$}
\caption{The Atiyah-Hirzebruch spectral sequence for $\widetilde\Omega^{\mathrm{fr}}_{10}(K(\Zz/3,3))$.}\label{fig:AHSS}
\end{figure}

Now we have a framed manifold $M$ together with a map $f\co M\to K(\Zz/3,3)$ such that $f_*[M]=m$. We still need the correct third homology group. By surgery below the middle dimension (see Kreck \cite{Kreck}*{Proposition 4}), $(M,f)$ is bordant to another manifold $(M',f')$ such that the new map $f'\co M'\to K(\Zz/3,3)$ is a 5-equivalence. Hence, $M'$ is 2-connected and $H_3(M')$ is isomorphic to $\Zz/3$. Since the bordism relation is understood over $K(\Zz/3,3)$, i.\,e.\ it takes the maps $f$ respectively $f'$ into account, the images of the fundamental classes $f_*[M]$ and $f'^*[M']$ are equal.
\end{proof}

\begin{cor}
There exists a simply-connected, closed, smooth, strongly chiral, 17"-dimensional manifold.
\end{cor}

\begin{proof}
The 10-dimensional manifold whose existence has just been shown has nonzero Betti numbers only in degrees 0, 10 and possibly 5. By \autoref{products2}, the product of this manifold with a 7-dimensional, strongly chiral rational homology sphere is strongly chiral.\qedhere

\vskip-\prevdepth
\hrule height0pt
\end{proof}

\subsection{Dimensions 9 and 13}\label{913}

In this section, examples of simply-connected, strongly chiral manifolds are given in the last two missing dimensions. Before we start proving their existence, some preliminaries are necessary.

As described above, we use the Postnikov tower of a simply-connected space. Two successive stages $P^n\to P^{n-1}$ of this tower form a principal fibration with a $K(\pi,n)$ as the fibre (see Griffiths and Morgan \cite{GM}*{\textsection\,VI.B}). This fibration is determined up to fibre homotopy equivalence by the $k$-invariant $k^{n+1}\in H^{n+1}(P^{n-1};\pi)$.

Given a principal $K(\pi,n)$-fibration $p\co E\to B$, it is known that the $k$-invariant $k^{n+1}$ is the image of the canonical element $\Delta\in H^n(K(\pi,n);\pi)=\Hom(\pi,\pi)$ under the transgression homomorphism which is the first possibly nonzero differential in the cohomology Serre spectral sequence with $\pi$-coefficients:
\[
  \xymatrix@R=1pt@C=6pt{
    \llap{$\displaystyle d_{n+1}\co{}$}E^{0,n}_{n+1} \ar[r] & E^{n+1,0}_{n+1} \\
    *[left]{{\cong}\strut} & *[left]{{\cong}\strut}\\
    H^n(K(\pi,n);\pi) & H^{n+1}(B;\pi)
    }
\]

We would like, however, to relate the $k$-invariant to the transgression in the Serre spectral sequence with \emph{integer} coefficients. We distinguish the transgressions for the various coefficient groups in the following by a subscript to $\tau$.

\begin{prop}\label{trans}
Suppose that $\pi$ is a finitely generated, free abelian group. Let $E \to B$ be a principal fibration with the fibre $F\simeq K(\pi,n)$. Assume that $B$ is homotopy equivalent to a CW-complex and $H_i(B)$ is finitely generated for $i\leq n+2$. Then the map
\[
 \xymatrix@=0pt{
 H^{n+1}(B;\pi) \ar@{}[rr]|-*{\to} && \Hom(H^n(F),H^{n+1}(B))\\
 \textup{$k$-invariant} & \mapsto & \textup{transgression in the Serre spectral sequence}\\
 k^{n+1}=\tau_\pi(\Delta) & \mapsto & (\tau_\Zz=d_{n+1}\co E^{0,n}_{n+1}\to E^{n+1,0}_{n+1})
 }
\]
coincides with the chain of natural isomorphisms
\[\openup\jot\tabskip0pt plus 1000pt
 \halign to \displaywidth{\hfill$\displaystyle#{}$\tabskip0pt&$\displaystyle{}#$\hfill\tabskip0pt plus 1000pt&\hfill{\normalfont#}\tabskip0pt\cr
 \noalign{\vskip-\jot}
 H^{n+1}(B;\pi) & \to H^{n+1}(B;H_n(F))&(Hurewicz map)\cr
 & \leftarrow H^{n+1}(B)\otimes H_n(F) &(Universal coefficient theorem)\cr
 & \to H_n(F;H^{n+1}(B))               &(Universal coefficient theorem)\cr
 & \to\Hom(H^n(F),H^{n+1}(B)).         &(Universal coefficient theorem)\crcr
 }
\]
\end{prop}

Note that all relevant $\Ext$ and $\Tor$ groups in the universal coefficient theorems vanish because $H_{n-1}(F)=H^{n+1}(F)=0$ and $H_n(F)$ is finitely generated free.

\begin{proof}[Outline of proof]
A detailed proof is given in \cite{Muellner}*{Proposition 56}. The idea is to prove commutativity of the following diagram:
\[
 \xymatrix@C=4.5pc@R=2.2pc{
 \Hom(H^n(F),H^n(F)) \ar[r]^-{\Hom(\id,\tau_\Zz)} & \Hom(H^n(F),H^{n+1}(B)) \\
 H_n(F;H^n(F)) \ar[u]\ar[r]^{(\tau_\Zz)_*}_-{\text{coefficient change}} & H_n(F;H^{n+1}(B))\ar[u]\\
 H^n(F)\otimes H_n(F)\ar[u]\ar[d]\ar[r]^{\tau_\Zz\otimes\id} & H^{n+1}(B)\otimes H_n(F)\ar[u]\ar[d]\\
 H^n(F;H_n(F)) \ar[d]\ar[r]^{\tau_{H_n(F)}} & H^{n+1}(B;H_n(F))\ar@{.>}[d]\\
 \Hom(H_n(F),H_n(F)) \ar[r] & \Hom(H_{n+1}(B),H_n(F))
 }
\]

All vertical maps are parts of universal coefficient sequences, and are all isomorphisms except for the dotted arrow at the bottom right. Remembering the maps in the universal coefficient theorem and according to the known relation between the $k$-invariant and the transgression with $H_n(F)$-coefficients, elements are mapped in the following way:\pagebreak[0]
\[
 \xymatrix@C=1pc@R=0pt{
 \id & \mapsto & \tau_\Zz\\
 *[left]{{\mapsto}\strut} && *[left]{{\mapsto}\strut}\\
 \bullet && \bullet\\
 *[left]{{\mapsto}\strut} && *[left]{{\mapsto}\strut}\\
 \bullet && \bullet\\
 *[right]{{\mapsto}\strut} && *[right]{{\mapsto}\strut}\\
 \Delta & \mapsto & k^{n+1}\\
 *[right]{{\mapsto}\strut} \\
 \id
 }
\]

Thus, the maps are exactly as stated in \autoref{trans}.\qedhere

\vskip-\prevdepth
\hrule height0pt
\end{proof}

\subsubsection{The 9-dimensional example}

\begin{thm}\label{1conn9ex}
There exists a simply-connected, closed, smooth, strongly chiral, 9"-dimensional manifold.
\end{thm}

\begin{proof}[Summary of proof]
The obstruction to amphicheirality in the Postnikov tower is a combination of rational and integral information. First, we construct a candidate for the Postnikov approximation $P^4\to P^3\to P^2$ of the desired manifold $M$ together with a candidate for the image of the fundamental class $m\in H_{9}(P^4)$. We show that there are very few automorphisms of $H_2(P^3)$ that can be induced from a self-homotopy equivalence $P^3\to P^3$.

Let $P^4\loc\to P^3\loc\to P^2\loc$ be the corresponding rational Postnikov tower and denote by $m_Q$ the image of $m$ in $H_9(P^4\loc)$. We show that $m_Q$ cannot be reversed by a self-map of $P^4\loc$ that induces one of the above automorphisms on $H_2$ (tensored with $\Qq$).

A short bordism argument shows that there really is a 9-dimensional manifold $M$ together with a map $g\co M\to P^4$ inducing the correct image of the fundamental class, i.\,e.\ $g_*[M]=m$. By surgery, we alter $M$ to $M'$ so that $g'\co M'\to P^3$ is a 4-equivalence and $g\co M'\to P^4\to P^4\loc$ is rationally a 5-equivalence. Due to functoriality of the Postnikov approximations $P^3$ and $P^4\loc$, $M'$ is strongly chiral.
\end{proof}

\minisec{Construction and automorphisms of \texorpdfstring{$P^3$}{P\textasciicircum 3}}

We start with a candidate for the Postnikov tower of fibrations $P^4\to P^3\to P^2$ of the desired manifold $M$. As the base, we choose $P^2\cong K(U,2)$ with $U\cong\Zz^3$. We fix a basis $(a,b,c)$ of the dual group $U^\vee\coloneq\Hom(U,\Zz)$. Likewise, we let $V\cong\Zz^3$ and fix a basis $(A,B,C)$ of the dual group $V^\vee$. The space $P^3$ is defined as a principal fibration over $P^2$ with the fibre $K(V,3)$. By \autoref{trans}, there is a bijection between the possible $k$-invariants and the first differential in the Serre spectral sequence. This correspondence allows us to define the fibration by its transgression
\[
  \vbox{\openup\jot\halign{\hfill$\displaystyle#$&$\displaystyle#$\hfill\cr
  \tau\co V^\vee & {}\to S^2(U^\vee)\cr
  A & {}\mapsto bc,\quad B\mapsto 2ac,\quad C\mapsto 3ab.\crcr}}
\]

Here, we used that the base is homotopy equivalent to $(\CP^\infty)^3$, whose cohomology algebra is the polynomial algebra $\Zz[a,b,c]=S^*(U^\vee)$.

The cohomology of $P^3$ can be computed by the Serre spectral sequence. This yields the following cohomology groups:\pagebreak[0]
\[
\vbox{\hbox{$\displaystyle\begin{array}[b]{ccl}
\toprule
i & H^i(P^3) & \text{generators} \\\midrule
0 & \Zz & 1\\
1 & 0 \\
2 & U^\vee & a,b,c\\
3 & 0 \\
4 & \Zz^3\oplus\Zz/2\oplus\Zz/3 & a^2,b^2,c^2,ac,ab\\
5 & \Zz^2 & 2aA-bB, 3aA-cC\\
\bottomrule
\end{array}$}}
\]

\begin{lemma}\label{diaglemma}
Let $T\co P^3\to P^3$ be a homotopy equivalence. Then the induced map on $H^2$ is necessarily of the form
\begin{equation}\label{diag}
 \begin{pmatrix}
   \pm1 & 0 & 0 \\
   0 & \pm1 & 0 \\
   0 & 0 & \pm1
 \end{pmatrix}
\end{equation}
with respect to the basis $(a,b,c)$.
\end{lemma}

\begin{proof}
Since $P^2$ is an Eilenberg-MacLane space $K(U,2)$ and the projection $P^3\to P^2$ induces an isomorphism on $H^2$ with any coefficients, the map $T$ and the Postnikov fibrations can be complemented to a homotopy-commutative square:
\[
 \xymatrix{
 P^3\ar[r]^-T\ar[d] & P^3\ar[d]\\
 P^2 \ar[r] & P^2
 }
\]

By the homotopy lifting property of a fibration, the map $T$ is homotopic to a fibre-preserving map $T'$. This yields a restriction to the fibre, $T'_{|K(V,3)}$, in addition to the induced map on the base $K(U,2)$. For simplicity, we write the induced maps in cohomology simply as $T^*$. From the functoriality of the Serre spectral sequence, we get
\begin{equation}\label{trafo}
  T^*\tau(v)=\tau(T^*v)
\end{equation}
for every $v\in V^\vee$.

Express the induced map on $H^2(P^3)=U^\vee$ by a matrix
\[
 M\coloneq \begin{pmatrix}
   g&h&i\\
   k&l&m\\
   p&q&r
 \end{pmatrix}\in M(3\times 3;\Zz).
\]

By equation \eqref{trafo}, we have
\[
 \tau(T^*C)=T^*(\tau(C))=T^*(3ab)=3(ga+kb+pc)(ha+lb+qc).
\]

Since the right hand side is in the image of $\tau$, the coefficients of $a^2$, $b^2$ and $c^2$ must be zero, i.\,e.\ $gh=kl=pq=0$. Considering the images of $A$ and $B$ in the same manner, we obtain that in each row of $M$, the product of two arbitrary entries must vanish. Thus there is at most one nonzero entry in each row of $M$.

Since $M$ is a unimodular matrix, it must be the product of a permutation matrix and a diagonal matrix with eigenvalues $\pm1$. We want to show that the only possible permutation is the identity.

Suppose that the permutation is a transposition, e.\,g.\ $(a\leftrightarrow b)$. This would imply $\tau(T^*A)=T^*(\tau(A))=T^*(bc)=\pm ac$ but only multiples of $2ac$ are in the image of $\tau$. Likewise, the other transpositions $(b\leftrightarrow c)$ and $(a\leftrightarrow c)$ as well as the 3-cycles $(a\to b\to c)$ and $(c\to b\to a)$ are excluded.
\end{proof}

\minisec{Construction of \texorpdfstring{$P^4$}{P\textasciicircum 4} and \texorpdfstring{$m$}m}

The next Postnikov stage, $P^4$, is again constructed as a principal fibration. We choose the fibre as a $K(W,4)$ with $W\cong\Zz^2$ and a basis $\alpha,\beta$ of the dual group $W^\vee$. The $k$-invariant is again determined by the transgression, which is chosen as the isomorphism
\[
 \vbox{\openup\jot\halign{\hfill$\displaystyle#{}$&$\displaystyle{}#$\hfill\cr
  \tau\co W^\vee &\to H^5(P^3)\cr
  \alpha &\mapsto 2aA-bB,\quad \beta\mapsto 3aA-cC.\crcr
 }}
\]

Below, the spectral sequence for this fibration immediately shows that $H^5(P^4)=0$ and therefore $H_5(P^4;\Qq)=0$. This result is needed later in \autoref{middlesurgery}. Blank entries in the following diagram represent the trivial group.
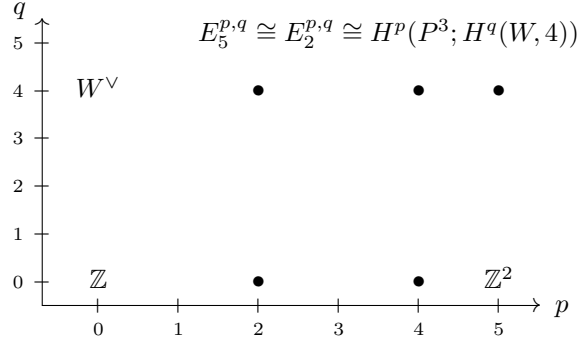
\[
 \xy
 <2.5pc,0pt>:<0pt,1.5pc>::
 (-.7,-.5);p, \ar (5.5,-.5)="id" \POS "id"+/r2\jot/ *!!<0pt,\mathaxis>!L{p},
 p, \ar (-.7,5.5)="id" \POS "id"+/l2\jot/ *!RD{q},
 \POS
 0 *!!<0pt,\mathaxis>+{\Zz}="00",
 (2,0) *!!<0pt,\mathaxis>+{\bullet}="20",
 (4,0) *!!<0pt,\mathaxis>+{\bullet}="40",
 (5,0) *!!<0pt,\mathaxis>+{\Zz^2}="50",
 (0,4) *!!<0pt,\mathaxis>+{W^\vee}="04",
 (0,0);(1,0) **@{},
 (0,-.5) *@{|}
 +(1,0) *@{|}
 +(1,0) *@{|}
 +(1,0) *@{|}
 +(1,0) *@{|}
 +(1,0) *@{|},
 (0,0);(0,1) **@{},
 (-.7,0) *@{|}
 +(0,1) *@{|}
 +(0,1) *@{|}
 +(0,1) *@{|}
 +(0,1) *@{|}
 +(0,1) *@{|},
 (2,4) *!!<0pt,\mathaxis>+{\bullet}="24",
 (4,4) *!!<0pt,\mathaxis>+{\bullet}="44",
 (5,4) *!!<0pt,\mathaxis>+{\bullet}="54",
 \POS
 (6,5.5) *!UR{E_5^{p,q}\cong E_2^{p,q}\cong H^p(P^3;H^q(W,4))},
 \let\objectstyle=\labelstyle
 \POS
 (0,-1) *{0}
 +(1,0) *{1}
 +(1,0) *{2}
 +(1,0) *{3}
 +(1,0) *{4}
 +(1,0) *{5}
 ,(-1,0) *{0}
 +(0,1) *{1}
 +(0,1) *{2}
 +(0,1) *{3}
 +(0,1) *{4}
 +(0,1) *{5}
 \endxy
\]

\begin{lemma}
There is a class $m\in H_9(P^4)$ such that
\begin{itemize}
\item $m$ is an element of infinite order,

\item
the image of $m$ in $H_9(P^4\loc)$ is never mapped to its negative under any self-map of $P^4\loc$ such that the induced map on $H^2(P^4\loc)$ is of the form \eqref{diag}.
\end{itemize}
\end{lemma}

By $P^4\loc$, we mean the rational localization of $P^4$, as described in Griffiths and Morgan \cite{GM}*{Chapter 7}. The above properties of $m$ obviously remain if $m$ is replaced by a nonzero multiple.

\begin{proof}
Consider the rational cohomology of $P^4\loc$. The minimal model for it (uniquely determined up to isomorphism) is the free, graded-commutative, rational differential graded algebra
\[
 \mathfrak M\coloneq \Qq[a',b',c',A',B',C',\alpha',\beta']
\]
with degrees $|a'|=|b'|=|c'|=2$, $|A'|=|B'|=|C'|=3$ and $|\alpha'|=|\beta'|=4$ and differentials
\[
 \vbox{\openup\jot\halign{\hfill$\displaystyle#$\hfill\cr
 da'=db'=dc'=0,\cr
 dA'=b'c',\quad dB'=2a'c', \quad dC'=3a'b',\cr
 d\alpha'=2a'A'-b'B',\quad d\beta'=3a'A'-c'C'.\crcr
 }}
\]

The generators are chosen so that $a'\in \mathfrak M^2$ maps to $a\in H^2(P^4\loc)$ under the isomorphisms
\[
  H^*(\mathfrak M) \cong H^*(P^4\loc;\Qq)\cong H^*(P^4;\Qq),
\]
and likewise for the other generators. These isomorphisms are natural with respect to self-maps of $P^4$. For the second isomorphism, this follows immediately from the universal property of a localization; see Griffiths and Morgan \cite{GM}*{Theorem 7.7 and Definition on page 90}. The naturality of the first isomorphism is proved in \cite{GM}*{Theorem 14.1}.

Consider the element $(d\alpha)\beta-ABC\in\mathfrak M^9$. It is easily verified that it is closed, thus it represents a cohomology class $\bar m_\Qq\in H^9(\mathfrak M)\cong H^9(P^4\loc)$. The cohomology class is nonzero since there is no expression in $\mathfrak M^8$ whose differential contains a summand $ABC$.

Let $m_\Qq\in H_9(P^4\loc)$ be a homology class such that $\langle\bar m_\Qq,m_\Qq\rangle\in\Qq$ is nonzero. The class $m_\Qq$ itself might not be in the image of $H_9(P^4)\to H_9(P^4\loc)$ but a nonzero multiple of $m_\Qq$ certainly is. We replace $m_\Qq$ by this multiple and choose a preimage $m\in H_9(P^4)$.

Now consider an automorphism of $\mathfrak M$. Note that the differentials in every Hirsch extension which is used to build $\mathfrak M$ are injective, i.\,e.\ $d$ is injective on the vector spaces $\Qq\{A,B,C\}$ and $\Qq\{\alpha,\beta\}$. For this reason, the automorphism of $\mathfrak M$ is completely determined by the restriction to the base degree \hbox{$\mathfrak M^2=\Qq\{a,b,c\}$}.

Let $T_a$ be the automorphism of $\mathfrak M^2$ which is given by
\[
 a\mapsto -a ,\quad b\mapsto b,\quad c\mapsto c.
\]

The automorphism $T_a$ extends uniquely to $\mathfrak M$ by
\[
 A\mapsto A ,\quad B\mapsto -B,\quad C\mapsto -C, \quad \alpha\mapsto -\alpha,\quad \beta\mapsto -\beta.
\]

It can be quickly checked that $T_a$ fixes $\bar m_\Qq$. Likewise, the automorphisms $T_b$ and $T_c$ which reverse $b$ respectively $c$ fix $\bar m_\Qq$. Hence, every automorphism $T$ of $P^4$ that induces a diagonal matrix of the form \eqref{diag} on $H^2(P^4\loc)\cong H^2(\mathfrak M)\cong \mathfrak M^2$ fixes $\bar m_\Qq$. Since the evaluation is natural, we have
\[
 \langle T_*m_\Qq,\bar m_\Qq\rangle = \langle m_\Qq,T^*\bar m_\Qq\rangle = \langle m_\Qq,\bar m_\Qq\rangle,
\]
so $m_\Qq$ cannot be reversed by $T$. The same clearly holds for $m$.
\end{proof}

\minisec{Bordism argument}

\begin{prop}\label{bord}
There is a framed, closed, smooth, 9-dimensional manifold $M$ together with a map $g\co M\to P^4$ such that $g_*[M]$ is a nonzero multiple of $m\in H_{9}(P^4)$.
\end{prop}

\begin{proof}
In analogy to the 10-dimensional case, it is sufficient to investigate the differentials starting from $E^k_{9,0}$ in the Atiyah-Hirzebruch spectral sequence for $\Omega^\fr_{*}(P^4)$. Since all coefficient groups $\Omega^\fr_i$ for $i>0$ are finite abelian groups, each of the finitely many differentials has a nonzero multiple of $m$ in the kernel, so a nonzero multiple of $m$ survives to the $E^\infty$-page.
\end{proof}

\minisec{Surgery}

In order to exploit the functoriality of the Postnikov towers, we aim to replace $M$ by surgery with a manifold $M'$ such that the corresponding map $g'\co M'\to P^4$ is a 4-equivalence and rationally a 5-equivalence. Note also that $M'$ is then automatically simply-connected.

Since $M$ is framed, its stable normal bundle $\nu\co M\to\BO$ is trivial. Thus, there is a lift of $\nu$ to the path space $\EO\simeq P\BO\simeq*$. Fix any such lift $\hat\nu\co M\to\EO$. Together with the map $g$ from the previous proposition, we use this to define a fibration and a lift
\[
 \xymatrix@!C=1.8pc{ & B\ar[d]^\xi && & P^4\times\EO \ar[d]^{\rlap{\scriptsize\vbox{\hbox{(projection to $\EO$,}\hbox{then end point map)}}}}\\
 M \ar[r]^\nu\ar[ru]^{\bar\nu} &\BO &\quad& M \ar[r]^\nu\ar[ru]^{g\times\hat\nu} &\BO.
 \POS "1,2";"2,4" **@{} ?<>(.5) *\txt{as}}
\]

The lift $\bar\nu$ is a \emph{normal $B$-structure} on $M$ in the language of Kreck \cite{Kreck}*{\textsection\,2}. By \cite{Kreck}*{Proposition 4}, $[M,g]$ is bordant over $P^4$ to $[M',g']$ such that $g'$ is a 4-equivalence. The proof of \autoref{1conn9ex} is completed by the following proposition. This is an extension of Kreck's surgery technique below the middle dimension to surgery on rational homology classes in the middle dimension. We give a detailed proof in \cite{Muellner}*{Proposition 64}.

\begin{prop}\label{middlesurgery}
Let $M'$ be an $m$-dimensional, closed, smooth, simply-connected manifold with normal $B$-structure $\bar\nu'\co M'\to B$ which is a $\left[\frac m2\right]$-equivalence. Assume that $m$ is odd and at least 5. Also assume that $H_{[m/2]+1}(B;\Qq)=0$. Then $(M',\bar\nu')$ can be replaced by a finite sequence of surgeries with $(M'',\bar\nu'')$ such that $\bar\nu''\co M''\to B$ is again a $\left[\frac m2\right]$-equivalence and additionally \hbox{$\pi_{[m/2]+1}(B,M'')\otimes\Qq=0$}.
\end{prop}

\subsubsection{Extension to dimension 13}

\begin{thm}
Let $M$ be a manifold as in the previous section with all described properties. The product $N\coloneq M\times\CP^2$ is a simply-connected closed, smooth, strongly chiral, 13-dimensional manifold.

\vskip-\prevdepth
\hrule height0pt
\end{thm}

\begin{proof}[Idea of proof]
The Postnikov tower P$^k_N$ of $N$ is very similar to the Postnikov tower of $M$. We have e.\,g.\ $P^2_N\simeq(\CP^\infty)^4$ and $P^3_N\simeq P^3\times\CP^\infty$. With an analogous computation to \autoref{diaglemma} it is proved that for a homotopy equivalence $T\co P^3_N\to P^3_N$, the induced map on $H^2$ is necessarily of the form
\[
 \begin{pmatrix}
   \pm1 & 0 & 0 & *\\
   0 & \pm1 & 0 & *\\
   0 & 0 & \pm1 & *\\
   0 & 0 & 0 & \pm1
 \end{pmatrix}
\]
with respect to a suitable basis $(a,b,c,x)$ for $H^2(P_N^3)$.

\lineskiplimit1pt 

The localization \smash{$P^4\Nloc$} has the minimal algebra $\mathfrak M\otimes \Qq[x]$, where $\mathfrak M$ is the rational minimal algebra of $P^4\loc$. The fundamental class of $N$ is detected by $((d\alpha)\beta-ABC)x^2\in H^{13}(P^4\Nloc)$ and cannot be reversed by any map with matrix form as above.
\end{proof}

\section{Strongly chiral manifolds in all bordism classes}\label{sec:bordism}

In this section, we finish the proof of \autoref{thm:A} by splitting it into three separate cases: \autoref{oddbordant} shows the existence of strongly chiral manifolds in all bordism classes in all odd dimensions $\geq 3$, \autoref{evenbordant} in the even dimensions $\geq 6$, and \autoref{4bordant} deals with dimension four and signature zero. Recall that the bordism classes in dimension four are detected by the signature (Milnor and Stasheff \cite{MS}*{Chapters 17, 19}), and a manifold with nonzero signature is strongly chiral.

\begin{prop}\label{oddbordant}
Given an aspherical, strongly chiral, closed, smooth, $n$-dimensional manifold, there are strongly chiral manifolds in every $n$-dimensional oriented bordism class.
\end{prop}

\begin{proof}
Denote the given manifold by $M$. By a surgery argument, every oriented bordism class in dimensions $\geq 2$ contains a simply-connected representative $N$. We claim that $M\connsum N$ does not admit a self-map of degree $-1$. The collapsing map $p\co M\connsum N\to M$ has degree 1 and induces an isomorphism on the fundamental group. Since $M$ is aspherical, it is the first Postnikov stage of $M\connsum N$. Since the Postnikov approximation is functorial up to homotopy and the image of the fundamental class $p_*[M\connsum N]=[M]$ cannot be mapped to its negative, $M\connsum N$ is chiral.

Since $N$ runs through all $n$-dimensional bordism classes, the connected sum $M\connsum N$ does as well, which proves the proposition.
\end{proof}

The aspherical manifolds assumed in \autoref{oddbordant} can for example be the explicit odd-dimensional, strongly chiral manifolds from \autoref{oddexsec}. One can also apply the proposition to the manifolds of Belolipetsky and Lubotzky \cite{BL}, thus finishing the alternative proof of \autoref{thm:A}. The proof given below with its separate treatment of dimension $4$ is still interesting since e.\,g.\ \autoref{4bordant} produces manifolds with finite fundamental group.

\begin{prop}[\cite{Muellner}*{Proposition 74}]\label{evenbordant}
In every even dimension $\geq 6$ and every oriented bordism class, there is a strongly chiral, connected representative.
\end{prop}

\begin{proof}
Instead of the $n$-dimensional manifold $M$ above use a product $M^{n-3}\times L^3$, where $L^3$ is a strongly chiral, 3-dimensional lens space and $M^{n-3}$ is one of the manifolds from \autoref{oddexsec}. Since strong chirality of lens spaces is detected by the cohomology structure, also the induced image of the fundamental class $[L]$ under the inclusion $L^3\subset L^\infty$ to the corresponding infinite-dimensional lens space cannot be reversed by a self-map. By examining the possible endomorphisms of the fundamental group, it can be shown that with our choice of $M$ from \autoref{oddexsec}, the image of the fundamental class $[M^{n-3}\times L^3]$ can never be reversed in the homology of the first Postnikov approximation $M^{n-3}\times L^\infty$. Therefore, by a very similar argument as in the previous lemma, all connected sums $(M\times L)\connsum N$ for simply-connected manifolds $N$ are strongly chiral.
\end{proof}

It remains to prove that there are strongly chiral, 4-dimensional manifolds with signature zero. Since every simply-connected, closed 4-manifold with signature zero is topologically amphicheiral, such a manifold must certainly have a nontrivial fundamental group. We use again the idea that the obstruction to amphicheirality should already be manifest in the 1-type, as it was in the two preceding propositions.

\begin{prop}\label{4exmp}
Let $\pi$ be a finite group such that every automorphism of $\pi$ is an inner automorphism and there is an element $m\in H_4(\pi)$ of order greater than two. Then there is a closed, connected, smooth, strongly chiral 4-manifold with fundamental group $\pi$ and signature equal to any given value.
\end{prop}

\begin{proof}
In analogy to the arguments for \autoref{10dim1conn} and \autoref{1conn9ex}, consider the oriented bordism group $\Omega^{\SO}_4(K(\pi,1))$. The Atiyah-Hirzebruch spectral sequence below shows that there is no differential from or to $E^2_{4,0}$. Hence, the Thom homomorphism $\Omega^{\SO}_{4}(K(\pi,1)) \to H_{4}(K(\pi,1))$ is surjective. Let $(M',f')$ be a preimage of $m$.
\[
 \xy
 <2.3pc,0pt>:<0pt,1.6pc>::
 (-.5,-.5);p \ar (4.7,-.5)="id"
   \POS "id"+/r2\jot/ *!!<0pt,\mathaxis>!L{r},
 p \ar (-.5,4.5)="id"
   \POS "id"+/l2\jot/ *h!!<0pt,\mathaxis>!R{s},
 0 *!!<0pt,\mathaxis>{\Zz},
 +(1,0) *!!<0pt,\mathaxis>{\pi\ab},
 +(1,0) *!!<0pt,\mathaxis>{*},
 +(1,0) *!!<0pt,\mathaxis>+{*} ="30",
 +(1,0) *!!<0pt,\mathaxis>+{H_4(\pi)} ="40",
 "30";
 (2,1) *!!<0pt,\mathaxis>+{0} **\dir{-};
 (1,2) *!!<0pt,\mathaxis>+{0} **\dir{}="12" ?<;?>>> **\dir{-},
 "12";
 (0,3) *!!<0pt,\mathaxis>+{0} **\dir{-},
 (0,2) *!!<0pt,\mathaxis>{0},
 (0,1) *!!<0pt,\mathaxis>{0},
 (0,4) *!!<0pt,\mathaxis>{\Zz},
 (5,4) *h!UR{E^2_{r,s}\cong H_r(K(\pi,1);\Omega^{\SO}_s)},
 "40"!R(.5) \ar_-{d_i} "12"
 \endxy
\]

By surgery below the middle dimension, $(M',f')$ can be altered to $(M,f)$ in the same bordism class such that $f\co M\to K(\pi,1)$ is a 2-equivalence; see Kreck \cite{Kreck}*{Proposition 4}. The signature of $M$ can then be corrected to any value by taking the connected sum with several copies of $\CP^2$ or $-\CP^2$. The map $f$ is tacitly precomposed with the collapsing map $M\connsum \pm\CP^2\to M$, which is also a 2-equivalence.

The map $f$ is a first Postnikov approximation for $M$, and so every homotopy equivalence of $M$ induces an automorphism of $K(\pi,1)$. Since every automorphism of $\pi$ is inner and inner automorphisms induce the identity on group homology (see Brown \cite{Brown}*{Proposition II.6.2}), $m=f_*[M]$ is fixed under any automorphism of $\pi$. Thus, since $m\neq -m$, the fundamental class $[M]$ can never be sent to its negative under any homotopy equivalence of $M$. Since manifolds with finite fundamental group are \Hopfian{}, every map $M\to M$ of degree $\pm 1$ is a homotopy equivalence.
\end{proof}

In the rest of this section, we present an infinite set of finite groups that fulfill the requirements of \autoref{4exmp}. For this, a certain family of finite groups is studied, which all have only inner automorphisms. An infinite subset of these groups also fulfills the condition on the fourth homology group (\autoref{HG}).

Let $p_1,\ldots,p_k$ be pairwise distinct odd primes. Let $G_i$ be the finite split metacyclic group $G_i\coloneq \Zz/p_i\rtimes \Zz/(p_i-1)$ defined by an isomorphism $\Zz/(p_i-1)\cong\Aut(\Zz/p_i)$.

Every automorphism of the product $G_1\times\ldots\times G_k$ is inner. This extends an example in Huppert's book \cite{Huppert}*{Beispiel I.4.10}, where the case of a single factor is proved. For the author's full proof, inspired by the cited reference, see \cite{Muellner}*{Proposition 77}.

We now turn to the homology groups of $G_1\times\ldots\times G_k$ with constant integral coefficients. Wall computed in \cite{Wall61} the integral homology of finite split metacyclic groups. Applying the Künneth theorem to this computation, we easily get the following result:

\begin{prop}[\cite{Muellner}*{Proposition 80}]\label{HG}
The group $G_1\times\ldots\times G_k$ has an element of order greater than 2 in $H_4(G)$ if and only if there are indices $i,j\in\{1,\ldots,k\}$, $i\neq j$ such that either $p_i=3$ and $p_j\equiv 1$ \textup{mod 3} or $\gcd(p_i-1,p_j-1)>2$.
\end{prop}

In summary, we proved that the following theorem:
\begin{thm}\label{4bordant}
There are infinitely many (with finite fundamental groups of different order) closed, connected, smooth, strongly chiral 4-manifolds with the signature equal to any given value.
\end{thm}

\section{Orientation-reversing diffeomorphisms of minimal order}\label{sec:minimalorder}

Another facet in the study of orientation reversal is the following question:
\begin{emptythm}
If a manifold is smoothly amphicheiral, what is the minimal order of an orientation-reversing diffeomorphism?
\end{emptythm}
Siebenmann presented a 3-manifold that admits an orientation-reversing diffeomorphism but none of finite order \cite{Siebenmann}*{page 176}. Another example is obtained by combining two theorems: Kreck proved in \cite{Kreck09} that there are infinitely many closed, simply-connected, smooth 6-manifolds on which no finite group can act effectively. However, \autoref{thm:B} asserts that every such manifold is smoothly amphicheiral. We can conclude:
\begin{prop}
There are infinitely many closed, simply-connected, smooth 6-manifolds which admit an orientation-reversing diffeomorphism but none of finite order.
\end{prop}

In view of diffeomorphisms of finite order, let $f\co M\to M$ be an orientation-reversing diffeomorphism of order $2^k\cdot l$ with $l$ odd. Then $f^l$ is an orientation-reversing diffeomorphism of order $2^k$. Thus, only powers of two are relevant for the minimal order of an orientation-reversing diffeomorphism, and we ask the following question:
\begin{emptythm}
Given $k>1$, is there a manifold which admits an orientation-reversing diffeomorphism of order $2^k$ but none of order $2^{k-1}$?
\end{emptythm}

This question can be answered in the affirmative. The key to one direction is that the possible degrees of self-maps of lens spaces are well-known. Indeed, let $L$ be a lens space of dimension $2n-1$ with fundamental group $\Zz/r$. There is a self-map of $L$ with degree $d\in\Zz$ inducing the endomorphism $x\mapsto e\cdot x$ on the fundamental group if and only if $e^n\equiv d$ \textup{mod} $r$, as proved by Olum \cite{Olum}*{Theorem~V}. We get the following statement as a corollary.

\begin{lemma}\label{lensmin}
A lens space of dimension $2n-1$ with fundamental group of order $r>2$ does not admit a self-map of degree $-1$ whose order is a divisor of $n$.
\end{lemma}

As immediate examples, this shows that no 3-dimensional lens space admits an orientation-reversing involution and no 7-dimensional lens space admits an orientation-reversing diffeomorphism of order less than 8.

To complement \autoref{lensmin}, we construct lens spaces with orientation-reversing diffeomorphisms of minimal order. Let $L$ be a lens space with prime fundamental group $\Zz/p$, $p\geq5$, and dimension $p-2$. Let $p-1=2^k\cdot l$ be the factorization into even and odd parts. By \autoref{lensmin}, $L$ has no orientation-reversing diffeomorphism of order $2^{k-1}$.

Since $p$ is prime, the group of multiplicative units in $\Zz/p$ is a cyclic group of order $p-1$. Let $c\in\Zz/p$ be a primitive root mod $p$, i.\,e.\ a generator of this group. In the following, abbreviate $(p-1)/2$ by $n$. Consider the lens space $L\coloneq L_p(c,c^2,\ldots,c^n)$. We follow the convention that $L$ is formed as the quotient of the unit sphere $S^{p-2}\subset\Cc^n$ under the $\Zz/p$-action
\[
   (z_1,\ldots,z_n)  \mapsto\left(\exp\left(\tfrac{2\pi ic}p\right)\cdot z_1,\ldots,\exp\left(\tfrac{2\pi ic^n}p\right)\cdot z_n\right).
\]

The diffeomorphism
\[
 \vbox{\openup\jot\halign{\hfill$\displaystyle#{}$&$\displaystyle{}#$\hfill\cr
  \widetilde f\co S^{p-2} &\to  S^{p-2} \cr
  \mskip-60mu(z_1,z_2,\ldots,z_n) & \mapsto(z_2,\ldots,z_n,\overline{z_1})\mskip-60mu\crcr
 }}
\]
preserves the $\Zz/p$-orbits. (Here we use that $c^n\equiv-1$ mod $p$ because both sides  of the equation are the unique element of order two in the cyclic group of units.) Moreover, $\widetilde f$ reverses the orientation, so it induces an orientation-reversing diffeomorphism $f$ on the lens space $L$. It follows from the definition that $\widetilde f^{2n}=\widetilde f^{p-1}=\id$. This implies that $f^l$ is an orientation-reversing diffeomorphism of order $2^k$.

By Dirichlet's theorem, the arithmetic progression
\[
 2^k+1,\quad 3\cdot 2^k+1,\quad 5\cdot 2^k+1,\quad\ldots
\]
contains infinitely many primes. Thus, for every positive integer $k$, there are suitable primes $p$, and this finally proves \autoref{thm:C}.

\vspace{1cm}

\parindent0pt\parskip\medskipamount\raggedright
\textsc{Hausdorff Research Institute for Mathematics, Poppelsdorfer Allee 82, 53115 Bonn, Germany}\par
E-mail: \href{mailto:muellner@math.uni-bonn.de}{\texttt{muellner@math.uni-bonn.de\vphantom{/p}}}\par
\url{http://www.math.uni-bonn.de/people/muellner}
\end{document}